\documentclass{amsart}

\usepackage[utf8]{inputenc}
\usepackage{verbatim}
\usepackage{amssymb}
\usepackage{amscd}
\usepackage{amsmath}
\usepackage[usenames]{color}
\newtheorem{theorem}{Theorem}[section]

\newtheorem{cor}[theorem]{Corollary}
\newtheorem{lemma}[theorem]{Lemma}
\newtheorem{proposition}[theorem]{Proposition}
\numberwithin{equation}{subsection}

\newcommand{\A}{\alpha}

\title{
A cotangent sum related to zeros of the Estermann zeta function}
\author{Michael Th. Rassias}
\date{\today}
\address{Department of Mathematics, ETH-Z\"{u}rich, R\"{a}mistrasse 101, 8092 Z\"{u}rich, Switzerland.}
\email{michail.rassias@math.ethz.ch}
\thanks{}

\begin{document}

 \maketitle

\begin{abstract}

We consider a cotangent sum related to Estermann's Zeta function. We provide an elementary and self-contained improvement of the error term in an asymptotic formula proved by V. I. Vasyunin. 
%\textbf{2000 Mathematics Subject Classification:} 65B10,\,\,65B15,\,\, 40Axx,\,\, 40Cxx,\,\,11L03,\,\,11M06.%
\newline

\end{abstract}

\section{Introduction}
The problem of the computation of zeros of the Estermann zeta function at $s=0$ reduces to the calculation of a specific cotangent sum. However, before we proceed with the presentation of our result related to that cotangent sum, we will exhibit a general framework which motivated the research in this area and explain how this problem is related to fundamental open problems, such as the Riemann Hypothesis.\\
The Estermann zeta function $E\left(s,\frac{h}{k},\alpha\right)$ is defined by the Dirichlet series
$$E\left(s,\frac{h}{k},\alpha\right)=\sum_{n\geq 1}\frac{\sigma_{\alpha}(n) \exp\left(2\pi ihn/k\right)}{n^s}\:, $$
where $h,k$ are positive integers, $(h,k)=1$, $Re\:s> \max\left(1, 1+ Re\:\alpha\right)$, $k\geq 1$ and
$$\sigma_{\alpha}(n)=\sum_{d|n}d^{\alpha}\:.$$
It is a known fact (see \cite{EST}, \cite{ishi2}) that the Estermann zeta function can be continued analytically to a meromorphic function, on the whole complex plane up to two simple poles $s=1$ and $s=1+\alpha$ if $\alpha\neq0$ or a double pole at $s=1$ if $\alpha=0$.\\
In 1995, Ishibashi (see \cite{ISH}) presented a nice result concerning the value of $E\left(s,\frac{h}{k},\alpha\right)$ at $s=0$.
\begin{theorem}(Ishibashi)
Let $k\geq 2$, $1\leq h<k$, $(h,k)=1$, $\alpha\in\mathbb{N}\cup\{0 \}$. Then\\
(1) For even $\A$, it holds
$$	E\left(0,\frac{h}{k},\alpha\right)=\left(-\frac{i}{2}\right)^{\alpha+1}\sum_{m=1}^{k-1}\frac{m}{k}\cot^{(\alpha)}\left(\frac{\pi mh}{k}\right)+\frac{1}{4}\delta_{\alpha,0}\:,$$
where $\delta_{\alpha,0}$ is the Kronecker delta function. \\
(2) For odd $\A$, it holds
$$E\left(0,\frac{h}{k},\alpha\right)=\frac{B_{\A+1}}{2(\A+1)}\:.$$
In the case when $k=1$ we have
$$E\left(0,1,\alpha\right)=\frac{(-1)^{\A+1}B_{\A+1}}{2(\A+1)}\:,$$
where by $B_m$ we denote the $m$-th Bernoulli number.
\end{theorem}
\noindent Hence for $k\geq 2$, $1\leq h<k$, $(h,k)=1$, it follows that
$$E\left(0,\frac{h}{k},0\right)=\frac{1}{4}+\frac{i}{2}c_0\left(\frac{h}{k}\right)\:,$$
where
$$c_0\left(\frac{h}{k}\right)=-\sum_{m=1}^{k-1}\frac{m}{k}\cot\left(\frac{\pi mh}{k}\right)\:.$$
This is exactly the cotangent sum that we are going to investigate in this paper. A very natural question that rises, is why this specific form of the Estermann zeta function is interesting. One example which demonstrates its importance is the following.\\
\noindent In 1985, R. Balasubramanian, J. Conrey and D. R. Heath-Brown (see \cite{bala}), used properties of $E\left(s,\frac{h}{k},\alpha\right)$ to prove an asymptotic formula for
$$I=\int_0^T\left|\zeta\left(\frac{1}{2}+it\right)\right|^2\left|A\left(\frac{1}{2}+it\right)\right|^2dt\:,$$
where $A(s)$ is a Dirichlet polynomial.\\
Asymptotics for functions of the form of $I$ are useful for theorems which provide a lower bound for the portion of zeros of the Riemann zeta-function on the critical line (see \cite{IWC}, \cite{IKW}). Period functions and families of cotangent sums appear in recent work of S. Bettin and J. B. Conrey (see \cite{BEC}). They generalize the Dedekind sum and share with it the property of satisfying a reciprocity formula. They have proved a reciprocity formula for the Vasyunin sum, which appears in the Nyman-Beurling criterion (see \cite{bag}) for the Riemann Hypothesis.\\
In 1995 Vasyunin (see \cite{VAS}), using Riemann sums, proved in a short elegant way the following theorem.
\begin{theorem}\label{x:vas}(Vasyunin)
For large integer values of $b$, we have
$$c_0\left(\frac{1}{b}\right)=\frac{1}{\pi}\:b\log b-\frac{b}{\pi}(\log 2\pi-\gamma)+O(\log b)\:,$$
where $\gamma$ is the Euler-Mascheroni constant.
\end{theorem}
\noindent In this paper we provide an elementary and self-contained improvement of the error term in Vasyunin's asymptotic formula. Namely, we prove the following:
\begin{theorem}
For large integer values of $b$, we have
$$c_0\left(\frac{1}{b}\right)=\frac{1}{\pi}\:b\log b-\frac{b}{\pi}(\log 2\pi-\gamma)+O(1)\:.$$
\end{theorem}

\section{Construction and some approximations for $c_0(1/b)$}
\vspace{5mm}
It is evident that for $a$, $b\in\mathbb{N}$, $b \geq 2$, the integer part $\left\lfloor a/b \right\rfloor$ is equal to the number of integers between $1$ and $a$ which are divisible by $b$. However, it is a basic fact that
\begin{equation}
\frac{1}{b}\sum_{m=0}^{b-1}e^{2\pi imk/b}=\left\{
\begin{array}{l l}
    1\:, & \quad \text{if}\ b\;|\;k\\
    0\:, & \quad \text{otherwise}\:.\\
  \end{array} \right.
 \nonumber
\end{equation}
Hence, it follows that
$$\left\lfloor \frac{a}{b}\right\rfloor=\frac{1}{b}\sum_{k=1}^{a}\sum_{m=0}^{b-1}e^{2\pi imk/b}\:.$$
But
$$\frac{e^{2\pi im/b}}{e^{2\pi im/b}-1}=\frac{1}{2}-\frac{i}{2}\cot\left(\frac{\pi m}{b}\right)\:,$$
since
$$\cot x=\frac{i(e^{ix}+e^{-ix})}{e^{ix}-e^{-ix}}\:.$$
So
$$\left\lfloor \frac{a}{b}\right\rfloor=\frac{L}{b}+\frac{1}{b}\sum_{m=1}^{b-1}\left(\frac{1}{2}-\frac{i}{2}\cot \left(\frac{\pi m}{b}\right)\right)\left(e^{2\pi ima/b}-1\right)\:,$$
where
$$L=\lim_{x\rightarrow0}\frac{e^{2\pi ixa/b}-1}{e^{2\pi ix/b}-1}=a\:.$$
Thus
\begin{eqnarray}
\left\lfloor \frac{a}{b}\right\rfloor&=&\frac{L}{b}-\frac{1}{b}\sum_{m=1}^{b-1}\left(\frac{1}{2}-\frac{i}{2}\cot \left(\frac{\pi m}{b}\right)\right)+\frac{1}{b}\sum_{m=1}^{b-1}\left(\frac{1}{2}-\frac{i}{2}\cot \left(\frac{\pi m}{b}\right)\right)e^{2\pi ima/b}\nonumber\\
&=&\frac{a}{b}-\frac{b-1}{2b}+\frac{i}{2b}\sum_{m=1}^{b-1}\cot \left(\frac{\pi m}{b}\right)+\frac{1}{b}\sum_{m=1}^{b-1}\left(\frac{1}{2}-\frac{i}{2}\cot \left(\frac{\pi m}{b}\right)\right)e^{2\pi ima/b}\nonumber\\
&=&\frac{a}{b}+\frac{1}{2b}-\frac{1}{2}+\frac{1}{2b}\sum_{m=1}^{b-1}\left(1-i\cot\left(\frac{\pi m}{b}\right)\right)e^{2\pi ima/b}\:,\nonumber
\end{eqnarray}
since
$$\sum_{m=1}^{b-1}\cot \left(\frac{\pi m}{b}\right)=0\:.$$
Hence, we obtain
\begin{eqnarray}
x_n:=\left\{\frac{na}{b}\right\}&=&\frac{na}{b}-\left\lfloor \frac{na}{b}\right\rfloor\nonumber\\
&=&\frac{na}{b}-\left(\frac{na}{b}+\frac{1}{2b}-\frac{1}{2}+\frac{1}{2b}\sum_{m=1}^{b-1}e^{2\pi imna/b}-\frac{i}{2b}\sum_{m=1}^{b-1}\cot\left(\frac{\pi m}{b}\right)e^{2\pi imna/b}\right)\:.\nonumber
\end{eqnarray}
(For some nice applications of fractional parts, see \cite{fur}.)\\
We can write
$$x_n=\left(\frac{1}{2}-\frac{1}{2b}-\frac{1}{2b}\sum_{m=1}^{b-1}e^{2\pi imna/b}\right)+\frac{i}{2b}\sum_{m=1}^{b-1}\cot\left(\frac{\pi m}{b}\right)e^{2\pi imna/b}\:.$$
Let
$$T=\frac{i}{2b}\sum_{m=1}^{b-1}\cot\left(\frac{\pi m}{b}\right)e^{2\pi imna/b}\:.$$
Then we get
\begin{eqnarray}
T&=&\frac{i}{2b}\sum_{m=1}^{b-1}\left[\cot\left(\frac{\pi m}{b}\right)\cos\left(2\pi mn\frac{a}{b}\right)+i\cot\left(\frac{\pi m}{b}\right)\sin\left(2\pi mn\frac{a}{b}\right)\right]\nonumber\\
&=&-\frac{1}{2b}\sum_{m=1}^{b-1}\cot\left(\frac{\pi m}{b}\right)\sin\left(2\pi mn\frac{a}{b}\right)+\frac{i}{2b}\sum_{m=1}^{b-1}\cot\left(\frac{\pi m}{b}\right)\cos\left(2\pi mn\frac{a}{b}\right)\:.\nonumber
\end{eqnarray}
Therefore,
\begin{eqnarray}
x_n&=&\left[\frac{1}{2}-\frac{1}{2b}-\frac{1}{2b}\sum_{m=1}^{b-1}e^{2\pi imna/b}-\frac{1}{2b}\sum_{m=1}^{b-1}\cot\left(\frac{\pi m}{b}\right)\sin\left(2\pi mn\frac{a}{b}\right)\right]\nonumber\\
&+&\frac{i}{2b}\sum_{m=1}^{b-1}\cot\left(\frac{\pi m}{b}\right)\cos\left(2\pi mn\frac{a}{b}\right)\:.\nonumber
\end{eqnarray}
But, since $x_n\in\mathbb{R}$ and
\begin{equation}
\sum_{m=1}^{b-1}e^{2\pi imna/b}=\left\{
\begin{array}{l l}
    -1\:, & \quad \text{if $b\nmid na$}\vspace{2mm}\\
    b-1\:, & \quad \text{otherwise}\:,\\
  \end{array} \right.
  \nonumber
\end{equation}
we obtain the following Proposition.

\begin{proposition}
For every $a$, $b$, $n\in\mathbb{N}$, $b \geq 2$, we have
$$\sum_{m=1}^{b-1}\cot\left(\frac{\pi m}{b}\right)\cos\left(2\pi mn\frac{a}{b}\right)=0\:.$$
If $b\nmid na$ then we also have
$$x_n=\frac{1}{2}-\frac{1}{2b}\sum_{m=1}^{b-1}\cot\left(\frac{\pi m}{b}\right)\sin\left(2\pi mn\frac{a}{b}\right)\:.$$
\end{proposition}
\noindent Thus, for every $a$, $b\in\mathbb{N}$, with $b\nmid a$, it holds
$$x_1=\frac{1}{2}-\frac{1}{2b}\sum_{m=1}^{b-1}\cot\left(\frac{\pi m}{b}\right)\sin\left(2\pi m\frac{a}{b}\right)\:.$$
Hence, we can write
\begin{eqnarray}
\sum_{\substack{a\geq 1 \\ b\nmid a}}\frac{b(1-2x_1)}{a}&=&\sum_{\substack{a\geq 1 \\ b\nmid a}}\frac{1}{a}\sum_{m=1}^{b-1}\cot\left(\frac{\pi m}{b}\right)\sin\left(2\pi m\frac{a}{b}\right)\nonumber\\
&=&\sum_{a\geq 1} \sum_{m=1}^{b-1}\cot\left(\frac{\pi m}{b}\right)\frac{\sin\left(2\pi m\frac{a}{b}\right)}{a}\:.\nonumber
\end{eqnarray}
However, since
\[
\sum_{a\geq 1}\frac{\sin(a\theta)}{a}=\frac{\pi-\theta}{2},\:0<\theta<2\pi\tag{S}
\]
we obtain the following proposition.
\begin{proposition}$\label{x:ena}$
For every positive integer $b$, $b\geq 2$, we have
\[
c_0\left(\frac{1}{b}\right)=\frac{1}{\pi}\sum_{\substack{a\geq1 \\ b\nmid a}}\frac{b(1-2\{a/b\})}{a}\:.\tag{1}
\]
\end{proposition}
\noindent If we substitute $\left\{a/b\right\}$ in (1) by $a/b-\left\lfloor a/b\right\rfloor$ and carry out the calculations, we can express $c_0(1/b)$ in the equivalent form
$$ c_0\left(\frac{1}{b}\right)=\frac{1}{\pi}\sum_{\substack{a\geq1 \\ b\nmid a}}\left[\frac{b}{a}\left(1+2\left\lfloor \frac{a}{b}\right\rfloor\right)-2 \right]$$
Set
$$G_L(b)=\sum_{\substack{1\leq a\leq L \\ b\nmid a}}\left(\frac{b}{a}\left(1+2\left\lfloor \frac{a}{b}\right\rfloor\right)-2\right),$$
then
\begin{eqnarray}
G_L(b)&=&\sum_{1\leq a\leq L}\left(\frac{b}{a}\left(1+2\left\lfloor \frac{a}{b}\right\rfloor\right)-2\right)-\sum_{\substack{1\leq a\leq L \\ b|a}}\left(\frac{b}{a}\left(1+2\left\lfloor \frac{a}{b}\right\rfloor\right)-2\right)\nonumber\\
&=&\sum_{1\leq a\leq L}\left(\frac{b}{a}\left(1+2\left\lfloor \frac{a}{b}\right\rfloor\right)-2\right)-b\sum_{\substack{1\leq a\leq L \\ b|a}}\frac{1}{a}.\nonumber
\end{eqnarray}
But,
\begin{eqnarray}
\sum_{\substack{1\leq a\leq L \\ b|a}}\frac{1}{a}&=&\frac{1}{b}+\frac{1}{2b}+\cdots+\frac{1}{\left\lfloor L/b\right\rfloor b}=\frac{1}{b}\sum_{1\leq k\leq\left\lfloor L/b\right\rfloor}\frac{1}{k}.\nonumber
\end{eqnarray}
Since for every positive real number $x$ it holds
\[
\sum_{1\leq n\leq x}\frac{1}{n}=\log x+\gamma+O\left( \frac{1}{x} \right)=\log x+O(1),\tag{2}
\]
it follows that
\[
\sum_{\substack{1\leq a\leq L \\ b|a}}\frac{1}{a}=\frac{1}{b}\log \left\lfloor \frac{L}{b}\right\rfloor +O\left(\frac{1}{b}  \right)
\]
Therefore, we obtain the following lemma.
\begin{lemma}$\label{x:G}$
For every $b$, $L\in\mathbb{N}$, with $b$, $L\geq 2$, it holds
\[
G_L(b)=-\log\frac{L}{b}+b(\log L+\gamma)-2L+2b\sum_{1\leq a\leq L}\frac{1}{a}\left\lfloor \frac{a}{b}\right\rfloor+
O\left(\frac{b}{L}  \right).
\]
\end{lemma}
\noindent We shall approximate the sum
$$S(L;b)=2b\sum_{1\leq a\leq L}\frac{1}{a}\left\lfloor \frac{a}{b}\right\rfloor$$
up to a constant error and hence improve the asymptotic approximation of $c_0\left(1/b\right)$ by replacing Vasyunin's error term $O(\log b)$ by $O(1)$.\\
\textbf{Remark.} In the sequel, we always assume that $b|L$.
\begin{lemma}$\label{x:A}$
\begin{eqnarray}
S(L;b)=2b\sum_{k\leq L/b}k\left(\log\frac{(k+1)b-1}{kb-1}+\frac{1}{2}F_1(k)-\frac{1}{12}F_2(k)+O\left( \frac{1}{k^4b^4} \right)\right) ,\nonumber
\end{eqnarray}
where
$$F_i(k)=\frac{1}{((k+1)b-1)^i}-\frac{1}{(kb-1)^i}\:.$$
\end{lemma}
\begin{proof}
\begin{eqnarray}
S(L;b)&=&2b\sum_{1\leq a\leq L}\frac{1}{a}\left\lfloor \frac{a}{b}\right\rfloor=2b\sum_{k\leq L/b}k\sum_{kb\leq a<(k+1)b}\frac{1}{a}\nonumber\\
&=&2b\sum_{k\leq L/b}k\left(\log\frac{(k+1)b-1}{kb-1}+\frac{1}{2}F_1(k)-\frac{1}{12}F_2(k)+\frac{1}{120}F_4(k)\pm\cdots\right) .\nonumber
\end{eqnarray}
This proves the lemma, since $F_4(k)=O(k^{-4}b^{-4})$.
\end{proof}
\begin{lemma}$\label{x:B}$
Let
$$r(b)=\sum_{k\geq 1}k\left( \log\frac{(k+1)b-1}{kb-1}-\frac{1}{k}+\frac{1}{2k^2}-\frac{1}{bk^2} \right).$$
There is an absolute constant $C_0$, such that
$$r(b)=C_0+O(b^{-1}),$$
when $b$ tends to infinity.
\end{lemma}
\begin{proof}
The function $r$ is differentiable with respect to $b$, with
\begin{eqnarray}
\frac{dr(b)}{db}=O(b^{-2}).\nonumber
\end{eqnarray}
Thus
\begin{eqnarray}
r(b)&=&r(2)+\int_2^{\infty}\frac{dr(t)}{dt}dt+O(b^{-1}),\nonumber\\
&=&C_0+O(b^{-1}),\nonumber
\end{eqnarray}
where
$$C_0=r(2)+\int_2^{\infty}\frac{dr(t)}{dt}dt.$$
The improper integral exists since
$$\frac{dr(b)}{db}=O(b^{-2}).$$
This completes the proof of the lemma.
\end{proof}
\begin{lemma}$\label{x:C}$ For large integer values of $k$ and $b$ we have

(i) $$\frac{1}{2}F_1(k)=-\frac{1}{2k^2b}+\frac{1}{2k^3b}-\frac{1}{k^3b^2}+O\left(\frac{1}{k^4b} \right)$$
(ii) $$-\frac{1}{12}F_2(k)=\frac{1}{6k^3b^2}-\frac{1}{4k^4b^2}+\frac{1}{2k^4b^3} +O\left(\frac{1}{k^5b^2} \right) ,$$
where $F_i(k)$ is defined as in Lemma $\ref{x:A}$.
\end{lemma}
\begin{proof}
(i)
\begin{eqnarray}
\frac{1}{2}F_1(k)&=&\frac{1}{2}\left(\frac{1}{kb}\cdot\frac{1}{1+\frac{b-1}{kb}}-\frac{1}{kb}\cdot\frac{1}{1-\frac{1}{kb}}\right)\nonumber  \\
&=&\frac{1}{2kb}\left( 1-\frac{b-1}{kb}+\left(\frac{b-1}{kb} \right)^2-\left(\frac{b-1}{kb} \right)^3+O\left(\frac{1}{k^4} \right) \right)\nonumber\\
&-&\frac{1}{2kb}\left(1+\frac{1}{kb}+\frac{1}{k^2b^2}+\frac{1}{k^3b^3}+O\left( \frac{1}{k^4b^4}\right) \right) \nonumber\\
&=&-\frac{1}{2k^2b}+\frac{1}{2k^3b}-\frac{1}{k^3b^2}+O\left(\frac{1}{k^4b} \right)\nonumber
\end{eqnarray}
(ii) By calculating the second derivative of the Taylor expansion of $\log(1+x)$ around $x=0$, it follows that
$$\frac{1}{(1+x)^2}=1-2x+3x^2-4x^3+O(x^4),\ \text{for}\ |x|<1.$$
Hence,
\begin{eqnarray}
&-&\frac{1}{12}F_2(k)=-\frac{1}{12}\left( \frac{1}{(kb)^2}\cdot\frac{1}{\left(1+\frac{b-1}{kb} \right)^2}-\frac{1}{(kb)^2}\cdot\frac{1}{\left(1-\frac{1}{kb} \right)^2}\right)\nonumber\\
&=&-\frac{1}{12(kb)^2}\left(1-2\frac{b-1}{kb}+3\left(\frac{b-1}{kb} \right)^2-4\left(\frac{b-1}{kb} \right)^3+O\left( \frac{1}{k^4}\right)  \right)\nonumber\\
&+&\frac{1}{12(kb)^2}\left(1+2\frac{1}{kb}+3\frac{1}{(kb)^2}+4\frac{1}{(kb)^3}+O\left( \frac{1}{k^4b^4}\right)  \right)\nonumber\\
&=&\frac{1}{6k^3b^2}-\frac{1}{4k^4b^2}+\frac{1}{2k^4b^3} +O\left(\frac{1}{k^5b^2} \right)\nonumber
\end{eqnarray}
\end{proof}
\noindent By the use of the previous results, we shall prove the following lemma.
\begin{lemma}$\label{x:D}$
We have
$$S(L;b)=2bC_0+2L+(1-b)\log\frac{L}{b}+(1-b)\gamma+O\left(\frac{b^2}{L} \right)+O(1).$$
\end{lemma}
\begin{proof}
\noindent By Lemmas $\ref{x:A}$, $\ref{x:B}$ and $\ref{x:C}$ we get
\begin{eqnarray}
S(L;b)&=&2b\sum_{k\leq L/b}k\left(\log\frac{(k+1)b-1}{kb-1}-\frac{1}{k}+\frac{1}{2k^2}-\frac{1}{bk^2} \right)\nonumber\\
&+&2b\sum_{k\leq L/b}k\left(\frac{1}{k}-\frac{1}{2k^2}+\frac{1}{bk^2} +\frac{1}{2}F_1(k)-\frac{1}{12}F_2(k)+O\left(\frac{1}{k^4b^4} \right)\right)\nonumber\\
&=&2b\sum_{k\leq L/b}k\left(\frac{1}{k}-\frac{1}{2k^2}+\frac{1}{bk^2}-\frac{1}{2k^2b}+\frac{1}{2k^3b}-\frac{1}{k^3b^2}+O\left(\frac{1}{k^4b} \right)+\frac{1}{6k^3b^2}+O\left(\frac{1}{k^4b^2} \right)\right)\nonumber\\
&+&2bC_0+O(1)\nonumber\\
&=&2bC_0+2L-b\left(\log\frac{L}{b}+\gamma+O\left(\frac{b}{L} \right)\right)+\log\frac{L}{b}+\gamma+O\left(\frac{b}{L} \right)+O(1)\nonumber\\
&=&2bC_0+2L+(1-b)\log\frac{L}{b}+(1-b)\gamma+O\left(\frac{b^2}{L} \right)+O(1).\nonumber
\end{eqnarray}
\end{proof}
\noindent By Lemma $\ref{x:G}$ and Lemma $\ref{x:D}$ we obtain the following proposition.
\begin{proposition}
For integer values of $b$, such that $b|L$, we have
$$G_L(b)=b\log b+2bC_0+O\left( \frac{b}{L} \right)+O\left(\frac{b^2}{L} \right)+O(1).$$
\end{proposition}
\noindent However, by the definition of $G_L(b)$ it is evident that
$$c_0\left(\frac{1}{b} \right)=\frac{1}{\pi}\lim_{L\rightarrow+\infty}G_L(b)$$
and thus by the previous proposition, we obtain the following theorem
\begin{theorem}
For integer values of $b$, we have
$$c_0\left(\frac{1}{b} \right)=\frac{1}{\pi}b\log b+\frac{2bC_0}{\pi}+O(1).$$
\end{theorem}
\noindent But, by Vasyunin's theorem, we know that for large integer values of $b$ it holds
$$c_0\left(\frac{1}{b} \right)=\frac{1}{\pi}b\log b-\frac{b}{\pi}(\log 2\pi-\gamma)+O(\log b).$$
Therefore,
$$2C_0=\gamma-\log 2\pi$$
and hence we obtain the following corollary.
\begin{cor}
For large integer values of $b$, we have
$$c_0\left(\frac{1}{b} \right)=\frac{1}{\pi}b\log b-\frac{b}{\pi}(\log 2\pi-\gamma)+O(1).$$
\end{cor}
\noindent The above corollary improves Vasyunin's Theorem $\ref{x:vas}$.\\ \\
\textbf{Acknowledgments.} I would like to thank my Ph.D. advisor Professor E. Kowalski, who proposed to me this area of research and for very helpful discussions. I would also like to thank Dr. H. Bui
as well as Professors H. Maier and L. Toth for reading the manuscript and for their useful remarks.

\end{document}